\documentclass[11pt]{amsart}

\usepackage{amsmath,amssymb,amsthm}
\usepackage[margin=1.15in]{geometry}
\usepackage[
  colorlinks=true,
  linkcolor=blue,
  citecolor=blue,
  urlcolor=blue,
  pagebackref=true
]{hyperref}

\renewcommand*{\backref}[1]{}
\renewcommand*{\backrefalt}[4]{%
  \ifcase #1\relax
  \or {\footnotesize (Cited on p.~#2.)}%
  \else {\footnotesize (Cited on pp.~#2.)}%
  \fi
}

\newtheorem{theorem}{Theorem}
\newtheorem{lemma}[theorem]{Lemma}
\newtheorem{corollary}[theorem]{Corollary}
\newtheorem{remark}[theorem]{Remark}
\numberwithin{equation}{section}

\newcommand{\R}{\mathbb{R}}
\newcommand{\Z}{\mathbb{Z}}
\newcommand{\T}{\mathbb{T}}
\newcommand{\e}{\mathrm{e}}
\newcommand{\cmod}[1]{\,(\mathrm{mod}\,#1)}
\newcommand{\thmref}[1]{\hyperref[#1]{Theorem~\ref*{#1}}}
\newcommand{\lemref}[1]{\hyperref[#1]{Lemma~\ref*{#1}}}
\newcommand{\corref}[1]{\hyperref[#1]{Corollary~\ref*{#1}}}
\newcommand{\remref}[1]{\hyperref[#1]{Remark~\ref*{#1}}}

\title[Dense finite Sidon sets on arithmetic progressions]
{Dense finite Sidon sets on arithmetic progressions}
\author{Yuchen Ding}
\address{(Yuchen Ding) $^1$School of Mathematics, Yangzhou University, Yangzhou 225002, People's Republic of China}
\address{$^2$HUN-REN Alfr\'ed R\'enyi Institute of Mathematics, Budapest, Pf. 127, H-1364 Hungary}
\email{ycding@yzu.edu.cn}


\begin{document}

\begin{abstract}
Let \(S\subset \{1,2,\ldots,n\}\) be a Sidon set with \(|S|=n^{1/2}+O(n^{1/2-\delta})\) for some fixed \(\delta>0\). 
This article provides the following expected asymptotic formula
\[
\sum_{\substack{a\in S\\ a\equiv r\cmod{m}}} a^\ell
=\frac{1}{m(\ell+1)}n^{\ell+1/2}
+o\left(n^{\ell+1/2}\right),
\]
where \(m\geq 1\), \(0\leq r<m\), and \(\ell\geq 0\) are three integers.
This removes the additional hypothesis in a previous residue class
asymptotic formula by the author. The proof uses the Fourier uniformity of
extremal Sidon sets due to Ortega and Prendiville. 
\end{abstract}

\maketitle

\section{Introduction}

A set \(S\subset \Z\) is called a Sidon set if all sums \(a+b\), with
\((a,b)\in S^2\) considered as unordered pairs, are distinct. Let \(S_n\)
be the largest cardinality of a Sidon subset of $\{1,2,\ldots, n\}$.
 The classical upper bound of Erdős and Turán~\cite{ErdosTuran}, together
with Singer's construction~\cite{Singer} and prime gap estimates of
Baker, Harman and Pintz~\cite{BHP}, gives the near square root
size estimate
\begin{equation}\label{eq:maximal-sidon-size}
S_n=n^{1/2}+O\left(n^{21/80}\right).
\end{equation}

Uniform distribution questions for dense Sidon sets have a long history.
Lindstr\"om~\cite{Lindstrom} proved well distribution in residue classes,
and Kolountzakis~\cite{Kolountzakis} obtained quantitative refinements
for the result of Lindstr\"om. Cilleruelo's work on gaps
in dense Sidon sets~\cite{Cilleruelo} is another important input in this
circle of ideas. In particular, it is used by Ortega and Prendiville in
the improved Fourier uniformity estimate recalled below.

In~\cite{Ding2021}, the author proved that if \(S\subset[1,n]\) is a
Sidon set with \(|S|=S_n\), then for every \(\varepsilon>0\),
\[
\sum_{a\in S}a
=\frac{1}{2}n^{3/2}
+O\left(n^{111/80+\varepsilon}\right).
\]
The same paper also considered sums in residue classes. For fixed
\(m\geq 1\) and \(0\leq r<m\), it was shown conditionally that
\[
\sum_{\substack{a\in S\\ a\equiv r\cmod{m}}}a
\sim \frac{1}{2m}n^{3/2}.
\]
The condition imposed there was a lower bound for
\(|S\cap[1,k]|\) throughout a long terminal range of $k$. As observed by
Balasubramanian and Dutta~\cite{BalasubramanianDutta}, this condition is
rarely compatible with the expected distribution of a dense Sidon set.
The final remark of~\cite{Ding2021} explicitly asks whether the asymptotic remains true in residue class situation after removing this additional
condition. One purpose of the present note is to answer this question
unconditionally, in fact for all fixed power sums.

The author's subsequent paper~\cite{Ding2023} developed these weighted
summation questions further. It proved, for
\[
S=\{a_1<a_2<\cdots<a_t\}\subset[1,n],
\]
that for every positive integer \(\ell\),
\[
\sum_{i=1}^t a_i^\ell
=\frac{t}{\ell+1}n^\ell
+O\left(n^{\ell+3/8}\right)
\]
whenever
\(
t\geq n^{1/2}-n^{1/4}.
\)
Moreover, for \(|S|=S_n\) the author proved
\[
\sum_{a\in S}a
=\frac{1}{2}n^{3/2}
+O\left(n^{221/160}\right),
\]
which slightly improves the error term in~\cite{Ding2021}. The same
paper also proved the almost all estimate
\[
\sum_{a\in S}a
=\frac{1}{2}n^{3/2}
+O\left(n^{11/8}\log n\right)
\]
for all \(n\leq N\), with at most
\[
O\left(\frac{N}{(\log N)^{7/19}}\right)
\]
exceptions, and established asymptotic formulae for more general
weighted sums
\[
\sum_{i=1}^t i^s a_i^\ell.
\]
One of the results below gives a
residue class analogue of these weighted sums.

Balasubramanian and Dutta~\cite{BalasubramanianDutta} later introduced
a direct formula for the \(m\)-th element of a dense finite Sidon set.
As consequences of their formula, they recovered some results of
\cite{Ding2021,Ding2023} by a different argument and improved the
almost all error term from
\cite{Ding2023}, using a prime gap estimate of
Heath-Brown~\cite{HeathBrown}. Precisely, they proved for every \(\varepsilon>0\),
\[
\sum_{a\in S}a
=\frac{1}{2}n^{3/2}
+O\left(n^{11/8}\right)
\]
for all \(n\leq N\), with at most
\[
O\left(N^{4/5+\varepsilon}\right)
\]
exceptions. Their observation about the conditional residue class theorem
in~\cite{Ding2021} is one motivation for the present note. We use the
same prime gap device below to sharpen the residue class error term for
almost all values of \(n\).

We revisit the residue class problem using the 
Fourier uniformity of extremal Sidon sets, due to Ortega and
Prendiville~\cite{OrtegaPrendiville}. This gives a short unconditional
proof of the residue class weighted sum asymptotic formula. In fact, the same
argument gives a more general statement on a dense Sidon set has
the expected density on every finite arithmetic progression.

For a Sidon set
\(S\subset[1,n]\), put
\begin{equation}\label{eq:def-phi}
\Phi(S,n)=
n^{1/2}
\left(
\left|1-\frac{|S|}{n^{1/2}}\right|
+n^{-1/4}
\right)^{1/2}.
\end{equation}
The new results are the following asymptotic formulae. 

\begin{theorem}\label{thm:main}
Let \(\ell\geq 0\) be fixed, and let \(S\subset[1,n]\) be a Sidon set.
Then, for every finite arithmetic progression \(P\subset[1,n]\),
\begin{equation}\label{eq:main-asymptotic}
\sum_{a\in S\cap P} a^\ell
=\frac{|S|}{n}\sum_{a\in P}a^\ell
+O\left(n^\ell\Phi(S,n)\log(2n)\right).
\end{equation}
The implied constant depends only on \(\ell\).
\end{theorem}

\begin{corollary}\label{cor:local-residue}
Let \(m\geq 1\), \(0\leq r<m\), and \(\ell\geq 0\) be fixed integers.
Let \(S\subset[1,n]\) be a Sidon set. Then, uniformly for all intervals
of integers \(I\subset[1,n]\),
\begin{equation}\label{eq:local-residue}
\sum_{\substack{a\in S\cap I\\ a\equiv r\cmod{m}}} a^\ell
=
\frac{|S|}{n}
\sum_{\substack{a\in I\\ a\equiv r\cmod{m}}} a^\ell
+O\left(n^\ell\Phi(S,n)\log(2n)\right).
\end{equation}
The implied constant depends only on \(\ell\) and \(m\).
\end{corollary}

For the maximal cases below, we recall the standard estimate
\[
S_n=n^{1/2}+O\left(n^{21/80}\right),
\]
already stated in \eqref{eq:maximal-sidon-size}. This is the only
additional input needed to pass from the general residue class formula
to maximal Sidon sets. The next corollary contains the unconditional
maximal Sidon set form of the residue class result asked for in the
final remark of~\cite{Ding2021}.

\begin{corollary}\label{cor:residue-effective}
Let \(m\geq 1\), \(0\leq r<m\), and \(\ell\geq 0\) be fixed integers.
Let \(S\subset[1,n]\) be a Sidon set. Then
\begin{equation}\label{eq:residue-effective}
\sum_{\substack{a\in S\\ a\equiv r\cmod{m}}} a^\ell
=\frac{|S|}{m(\ell+1)}n^\ell
+O\left(n^\ell\Phi(S,n)\log(2n)\right).
\end{equation}
The implied constant depends only on \(\ell\) and \(m\).
In particular, if \(|S|=S_n\), then
\[
\sum_{\substack{a\in S\\ a\equiv r\cmod{m}}} a
=\frac{1}{2m}n^{3/2}
+O\left(n^{221/160}\log (2n)\right),
\]
where the implied constant depends only on \(m\).
\end{corollary}

In particular, if \(|S|\sim n^{1/2}\) and
\(\Phi(S,n)\log(2n)=o(n^{1/2})\), then the expected asymptotic formula
in each fixed residue class follows. Note that \(\Phi(S,n)\log(2n)=o(n^{1/2})\) indeed happens if $|S|=n^{1/2}+O(n^{1/2-\delta})$ for some $\delta>0$.
Under the hypothesis
\(|S|\geq n^{1/2}-n^{1/4}\), the same corollary gives the error term
\(O\big(n^{\ell+3/8}\log (2n)\big)\), with the implied constant depending only on
\(\ell\) and \(m\).

The next theorem is a weighted generalization of \thmref{thm:main},
giving parallel estimates after weighting
\(a_i\in S\) by \(i^s a_i^\ell\).

\begin{theorem}\label{thm:rank-weighted-local}
Let \(m\geq 1\), \(0\leq r<m\), and \(s,\ell\geq 0\) be fixed
integers. Let
\[
S=\{a_1<a_2<\cdots<a_t\}\subset[1,n]
\]
be a Sidon set. Then, uniformly for all intervals of integers
\(I\subset[1,n]\),
\begin{equation}\label{eq:rank-local}
\sum_{\substack{1\leq i\leq t\\ a_i\in I\\ a_i\equiv r\cmod{m}}}
i^s a_i^\ell
=
\frac{t^{s+1}}{n^{s+1}}
\sum_{\substack{a\in I\\ a\equiv r\cmod{m}}}a^{s+\ell}
+O\left(t^s n^\ell\Phi(S,n)\log(2n)\right).
\end{equation}
The implied constant depends only on \(s\), \(\ell\), and \(m\).
In particular,
\begin{equation}\label{eq:rank-residue}
\sum_{\substack{1\leq i\leq t\\ a_i\equiv r\cmod{m}}}
i^s a_i^\ell
=
\frac{t^{s+1}}{m(s+\ell+1)}n^\ell
+O\left(t^s n^\ell\Phi(S,n)\log(2n)\right).
\end{equation}
\end{theorem}

\begin{corollary}\label{cor:maximal}
Let \(m\geq 1\), \(0\leq r<m\), and \(s,\ell\geq 0\) be fixed integers.
Let
\[
S=\{a_1<a_2<\cdots<a_t\}\subset[1,n]
\]
be a Sidon set with \(t=S_n\). Then
\begin{equation}\label{eq:maximal-rank-residue}
\sum_{\substack{1\leq i\leq t\\ a_i\equiv r\cmod{m}}}
i^s a_i^\ell
=\frac{1}{m(s+\ell+1)}n^{\ell+(s+1)/2}
+O\left(n^{\ell+s/2+61/160}\log (2n)\right).
\end{equation}
The implied constant depends only on \(s\), \(\ell\), and \(m\).
\end{corollary}

The following almost all theorem follows by combining
\thmref{thm:rank-weighted-local} with the prime gap input used by
Balasubramanian and Dutta~\cite{BalasubramanianDutta}.

\begin{theorem}\label{thm:maximal-almost-all}
Let \(N\geq 2\) and \(\varepsilon>0\). For all but
\(O(N^{4/5+\varepsilon})\) integers \(n\leq N\), the following assertion
holds: for any fixed integers \(m\geq 1\), \(0\leq r<m\), and
\(s,\ell\geq 0\), if
\[
S=\{a_1<a_2<\cdots<a_t\}\subset[1,n]
\]
is a Sidon set with \(t=S_n\),
then
\begin{equation}\label{eq:almost-all-rank-residue}
\sum_{\substack{1\leq i\leq t\\ a_i\equiv r\cmod{m}}}
i^s a_i^\ell
=\frac{1}{m(s+\ell+1)}n^{\ell+(s+1)/2}
+O\left(n^{\ell+s/2+3/8}\log n\right).
\end{equation}
The implied constant in the formula depends only on \(s\), \(\ell\), and
\(m\). The exceptional set constant may also depend on \(\varepsilon\).
In particular, for \(s=0\) and \(\ell=1\),
\[
\sum_{\substack{a\in S\\ a\equiv r\cmod{m}}} a
=\frac{1}{2m}n^{3/2}
+O\left(n^{11/8}\log n\right),
\]
where the implied constant depends only on \(m\).
\end{theorem}

I hope these results would be useful for further investigations of the Sidon sets in future.
\section{Proofs}

\subsection{Fourier input}

We use the following consequences of Ortega and Prendiville's 
Fourier uniformity theorem for extremal Sidon
sets~\cite[Theorem~6.3]{OrtegaPrendiville}.
We write \(\T=\R/\Z\), identifying it with \([0,1)\) in integrals. For a
set \(A\subset\Z\), let \(1_A\) denote its indicator function. For a
function \(G:\T\to\mathbb C\), put
\(\|G\|_\infty=\sup_{\alpha\in\T}|G(\alpha)|\). For a finitely
supported function \(f:\Z\to\mathbb C\), write
\[
\widehat f(\alpha)=\sum_{x\in\Z}f(x)\e^{2\pi i\alpha x}
\qquad (\alpha\in\T).
\]

\begin{lemma}\label{lem:fourier}
Let \(S\subset[1,n]\) be a Sidon set. Then
\[
\left\|
\widehat{1_S}
-\frac{|S|}{n}\widehat{1_{[1,n]}}
\right\|_\infty
\ll \Phi(S,n).
\]
\end{lemma}

\begin{proof}
Ortega and Prendiville proved
~\cite[Theorem~6.3]{OrtegaPrendiville} that
\begin{equation}\label{eq:op-improved-fourier}
\left\|
\widehat{1_S}
-\frac{|S|}{n}\widehat{1_{[1,n]}}
\right\|_\infty
\ll
n^{1/2}
\left(
\left|1-\frac{|S|}{n^{1/2}}\right|
+n^{-1/4}
\right)^{1/2}.
\end{equation}
This is precisely the claimed bound.
\end{proof}

\begin{lemma}\label{lem:progression-l1}
Let \(P\subset[1,n]\) be a finite arithmetic progression. Then
\[
\int_{\T}\big|\widehat{1_P}(\alpha)\big|\,{\rm d}\alpha\ll \log(2n).
\]
\end{lemma}

\begin{proof}
If \(P\) is empty or consists of one point, the estimate is immediate.
Otherwise, the result follows from the standard progression estimate of
Ortega and Prendiville~\cite[Lemma~3.2]{OrtegaPrendiville}, since
\(|P|\leq n\).
\end{proof}

\begin{lemma}\label{lem:progression-discrepancy}
Let \(S\subset[1,n]\) be a Sidon set. Then, uniformly for every finite
arithmetic progression \(P\subset[1,n]\),
\[
|S\cap P|=\frac{|S|}{n}|P|+O\left(\Phi(S,n)\log(2n)\right).
\]
The implied constant is absolute.
\end{lemma}

\begin{proof}
Put
\[
h=1_S-\frac{|S|}{n}1_{[1,n]}.
\]
Then
\[
|S\cap P|-\frac{|S|}{n}|P|
=
\sum_{x\in P}h(x)
=
\sum_{x\in\Z}1_P(x)h(x).
\]
For finitely supported functions on \(\Z\), Fourier orthogonality gives
\[
\sum_{x\in\Z}1_P(x)h(x)
=
\int_{\T}\widehat{1_P}(\alpha)\widehat h(-\alpha)\,{\rm d}\alpha.
\]
Since
\(\widehat h=\widehat{1_S}-\frac{|S|}{n}\widehat{1_{[1,n]}}\), we
obtain
\begin{align*}
\sum_{x\in\Z}1_P(x)h(x)
&=
\int_{\T}\widehat{1_P}(\alpha)
\left(
\widehat{1_S}
-\frac{|S|}{n}\widehat{1_{[1,n]}}
\right)(-\alpha)\,{\rm d}\alpha\\
&\le \left\|
\widehat{1_S}
-\frac{|S|}{n}\widehat{1_{[1,n]}}
\right\|_\infty\int_{\T}\big|\widehat{1_P}(\alpha)\big|{\rm d}\alpha.
\end{align*}
Combining \lemref{lem:progression-l1} with \lemref{lem:fourier} gives
\[
|S\cap P|-\frac{|S|}{n}|P|
\ll \Phi(S,n)\log(2n).
\]
The bound is uniform in \(P\), as required.
\end{proof}

\subsection{Proofs of the power sum estimates}
\begin{proof}[Proof of \thmref{thm:main}]
Let
\[
E=\Phi(S,n)\log(2n).
\]
Let \(P\subset[1,n]\) be a finite arithmetic progression. For
\(1\leq k\leq n\), put
\[
P(k)=P\cap[1,k],\qquad A_P(k)=|S\cap P(k)|,\qquad R_P(k)=|P(k)|.
\]
Since \(P(k)\) is again a finite arithmetic progression,
\lemref{lem:progression-discrepancy} gives, uniformly in \(k\),
\[
A_P(k)=\frac{|S|}{n}R_P(k)+O(E).
\]
For \(\ell\geq 1\), the Abel summation gives
\begin{align*}
\sum_{a\in S\cap P}a^\ell&=\sum_{k=1}^{n}\big(A_P(k)-A_P(k-1)\big)k^\ell\\
&=n^\ell A_P(n)
-\sum_{k=1}^{n-1}A_P(k)\left((k+1)^\ell-k^\ell\right).
\end{align*}
Substituting the estimate for \(A_P(k)\) into the summation, we get
\begin{align}\label{eq-new-1}
\sum_{a\in S\cap P}a^\ell=
\frac{|S|}{n}
\left(
n^\ell R_P(n)
-\sum_{k=1}^{n-1}R_P(k)\left((k+1)^\ell-k^\ell\right)
\right)
+O(n^\ell E).
\end{align}
Using again Abel's summation we get
\begin{align}\label{eq-new-2}
n^\ell R_P(n)
-\sum_{k=1}^{n-1}R_P(k)\left((k+1)^\ell-k^\ell\right)=\sum_{a\in P}a^\ell.
\end{align}
Combining (\ref{eq-new-1}) with (\ref{eq-new-2}), we obtain
\begin{align*}
\sum_{a\in S\cap P}a^\ell=
\frac{|S|}{n}
\sum_{a\in P}a^\ell
+O(n^\ell E).
\end{align*}
This proves \eqref{eq:main-asymptotic} for \(\ell\geq 1\), with the
implied constant depending only on \(\ell\). The case \(\ell=0\) is
exactly \lemref{lem:progression-discrepancy}.
\end{proof}

\begin{proof}[Proof of \corref{cor:local-residue}]
For an interval \(I\subset[1,n]\), the set
\[
\{a\in I:a\equiv r\cmod{m}\}
\]
is a finite arithmetic progression. Applying \thmref{thm:main} to this
progression gives the claim.
\end{proof}

\begin{proof}[Proof of \corref{cor:residue-effective}]
Apply \thmref{thm:main} to
\[
P=\{a\in[1,n]:a\equiv r\cmod{m}\}.
\]
The elementary estimate for sums of powers in a fixed residue class is
\begin{equation}\label{eq:residue-power-reference}
\sum_{\substack{1\leq a\leq n\\ a\equiv r\cmod{m}}}a^\ell
=\frac{1}{m(\ell+1)}n^{\ell+1}+O(n^\ell).
\end{equation}
Multiplying by \(|S|/n\), the extra error is
\(O(|S|n^{\ell-1})\), which is absorbed by
\(O(n^\ell\Phi(S,n)\log(2n))\), since \(|S|\leq n\) and
$$
\Phi(S,n)\ge n^{1/2}
\left(
n^{-1/4}
\right)^{1/2}=n^{3/8}.
$$ 
Hence, \thmref{thm:main} together with
\eqref{eq:residue-power-reference} implies
\eqref{eq:residue-effective} for \(\ell\geq 1\). The case
\(\ell=0\) follows similarly from
\(|P|=n/m+O(1)\).
If \(|S|=S_n\), then \eqref{eq:maximal-sidon-size} gives
\(|S|=n^{1/2}+O(n^{21/80})\).  Then
\[
\Phi(S,n)
\ll
n^{1/2}\left(n^{-19/80}+n^{-1/4}\right)^{1/2}
\ll n^{61/160}.
\]
Taking \(\ell=1\) in
\eqref{eq:residue-effective} and replacing \(|S|\) in the main term
therefore gives the desired result.
\end{proof}

\subsection{Proof of the weighted theorem and maximal corollary}

\begin{proof}[Proof of \thmref{thm:rank-weighted-local}]
Put
\[
E=\Phi(S,n)\log(2n)
\]
and
\[
P=\{a\in I:a\equiv r\cmod{m}\}.
\]
The set \(P\) is a finite arithmetic progression. For \(a_i\in S\), the
index \(i\) is the initial counting function of \(S\) at \(a_i\). Hence
\lemref{lem:progression-discrepancy}, applied to the interval
\([1,a_i]\), gives
\begin{equation}\label{eq:index-approx}
i=\frac{t}{n}a_i+O(E).
\end{equation}
If \(s=0\), the formula is exactly \corref{cor:local-residue}.
Assume now that \(s\geq 1\). Put \(y_i=t a_i/n\). Since
\(0\leq i,y_i\leq t\), we have
\[
i^s-y_i^s=(i-y_i)\sum_{j=0}^{s-1}i^{s-1-j}y_i^j.
\]
Together with \eqref{eq:index-approx}, this gives
\begin{equation}\label{eq:rank-power-approx}
i^s=\left(\frac{t}{n}a_i\right)^s+O_s(t^{s-1}E).
\end{equation}
Therefore
\begin{equation}\label{eq:rank-weighted-reduction}
\sum_{\substack{1\leq i\leq t\\ a_i\in I\\ a_i\equiv r\cmod{m}}}
i^s a_i^\ell
=
\frac{t^s}{n^s}
\sum_{a\in S\cap P}a^{s+\ell}
+O_s(t^s n^\ell E).
\end{equation}
Applying \corref{cor:local-residue} with \(\ell\) replaced by
\(s+\ell\), we get
\begin{equation}\label{eq:rank-local-power-input}
\sum_{a\in S\cap P}a^{s+\ell}
=
\frac{t}{n}\sum_{a\in P}a^{s+\ell}
+O_{s,\ell,m}(n^{s+\ell}E).
\end{equation}
Combining \eqref{eq:rank-weighted-reduction} and
\eqref{eq:rank-local-power-input} proves \eqref{eq:rank-local}.

Taking \(I=[1,n]\) and using
\begin{equation}\label{eq:rank-residue-power-reference}
\sum_{\substack{1\leq a\leq n\\ a\equiv r\cmod{m}}}a^{s+\ell}
=\frac{1}{m(s+\ell+1)}n^{s+\ell+1}+O_{s,\ell,m}(n^{s+\ell})
\end{equation}
gives \eqref{eq:rank-residue}. The extra term obtained after
multiplication by \(t^{s+1}/n^{s+1}\) is
\(O(t^{s+1}n^{\ell-1})\), which is absorbed by the stated error term.
\end{proof}

\begin{proof}[Proof of \corref{cor:maximal}]
By \eqref{eq:maximal-sidon-size}, if \(t=S_n\), then
\[
\Phi(S,n)
\ll
n^{1/2}\left(n^{-19/80}+n^{-1/4}\right)^{1/2}
\ll n^{61/160}.
\]
Substituting this bound into \thmref{thm:rank-weighted-local}, with
\(I=[1,n]\), gives
\begin{equation}\label{eq:maximal-rank-before-t-replacement}
\sum_{\substack{1\leq i\leq t\\ a_i\equiv r\cmod{m}}}
i^s a_i^\ell
=
\frac{t^{s+1}}{m(s+\ell+1)}n^\ell
+O\left(n^{\ell+s/2+61/160}\log n\right).
\end{equation}
Finally,
\begin{equation}\label{eq:maximal-t-power}
t^{s+1}=n^{(s+1)/2}+O_s\left(n^{s/2+21/80}\right),
\end{equation}
again by \eqref{eq:maximal-sidon-size}. This replacement error is
absorbed by the error term in
\eqref{eq:maximal-rank-before-t-replacement}. Combining
\eqref{eq:maximal-rank-before-t-replacement} and
\eqref{eq:maximal-t-power} gives \eqref{eq:maximal-rank-residue}.
\end{proof}

\subsection{Proof of the almost all theorem}

\begin{lemma}\label{lem:almost-all-size}
Let \(N\geq 2\) and \(\varepsilon>0\). For all but
\(O(N^{4/5+\varepsilon})\) integers \(n\leq N\), one has
\[
S_n=n^{1/2}+O\left(n^{1/4}\right).
\]
The exceptional set constant depends on \(\varepsilon\), and the implied
constant in the displayed estimate is absolute.
\end{lemma}

\begin{proof}
The upper bound \(S_n\leq n^{1/2}+O(n^{1/4})\) is the classical theorem
of Erdős and Turán. We prove the corresponding lower bound for almost all
\(n\), following the prime gap argument used by Balasubramanian and
Dutta~\cite{BalasubramanianDutta}.

Let \(p_j\) be the \(j\)-th prime and put \(g_j=p_{j+1}-p_j\).
Heath-Brown's theorem~\cite[Theorem~1]{HeathBrown} gives, for every
\(\varepsilon>0\),
\begin{equation}\label{eq:heath-brown-large-gaps}
\sum_{\substack{p_j\leq x\\ g_j\geq p_j^{1/2}}}g_j
\ll x^{3/5+\varepsilon},
\end{equation}
with the implied constant depending on \(\varepsilon\). If
\[
p_j^2-1\leq n<p_{j+1}^2-1
\]
and \(g_j<p_j^{1/2}\), Singer's construction gives \(S_n\geq p_j\), and
also
\[
n^{1/2}-p_j< p_{j+1}-p_j=g_j<p_j^{1/2}\ll n^{1/4}.
\]
Thus \(S_n\geq n^{1/2}-O(n^{1/4})\) for all such \(n\).

It remains to count the \(n\leq N\) lying in intervals belonging to a bad
gap \(g_j\geq p_j^{1/2}\). Their number is at most
\[
\sum_{\substack{p_j\leq \sqrt{N+1}\\ g_j\geq p_j^{1/2}}}
\left(p_{j+1}^2-p_j^2\right)
\ll
\sqrt N
\sum_{\substack{p_j\leq \sqrt{N+1}\\ g_j\geq p_j^{1/2}}}g_j
\ll N^{4/5+\varepsilon},
\]
where Bertrand's postulate is used in the middle estimate and
\eqref{eq:heath-brown-large-gaps} is used in the final estimate. The
final implied constant depends on \(\varepsilon\). This proves the lemma.
\end{proof}

\begin{proof}[Proof of \thmref{thm:maximal-almost-all}]
For all \(n\leq N\) outside the exceptional set in
\lemref{lem:almost-all-size}, a maximal Sidon set satisfies
\[
\left|1-\frac{|S|}{n^{1/2}}\right|
\ll n^{-1/4}.
\]
Hence \(\Phi(S,n)\ll n^{3/8}\). Applying
\thmref{thm:rank-weighted-local}, with \(I=[1,n]\), gives
\begin{equation}\label{eq:almost-all-rank-before-t-replacement}
\sum_{\substack{1\leq i\leq t\\ a_i\equiv r\cmod{m}}}
i^s a_i^\ell
=
\frac{t^{s+1}}{m(s+\ell+1)}n^\ell
+O\left(n^{\ell+s/2+3/8}\log n\right).
\end{equation}
Since \(t=n^{1/2}+O(n^{1/4})\), we may replace \(t^{s+1}\) by
\(n^{(s+1)/2}\) in the main term. The error term introduced by this
replacement is \(O_s(n^{\ell+s/2+1/4})\), and is absorbed by the
error term in \eqref{eq:almost-all-rank-before-t-replacement}. This
proves \eqref{eq:almost-all-rank-residue}, and the last display in the
statement is the case \(s=0\) and \(\ell=1\).
\end{proof}

\begin{remark}
The proof gives more than the displayed asymptotic formula. It shows that
both the power sum problem in residue class and its weighted analogue
are controlled by a uniform estimate for the discrepancy of \(S\) on
finite arithmetic progressions. Any improvement in the Fourier uniformity
exponent for extremal Sidon sets would immediately improve the error term
in \thmref{thm:main} and \thmref{thm:rank-weighted-local}. The almost all
improvement in \thmref{thm:maximal-almost-all} instead comes from
improving the size defect of maximal Sidon sets for almost all \(n\).
\end{remark}

\section*{AI Disclosure}
During the preparation of this manuscript, the author used OpenAI Codex as an auxiliary tool. The tool was used to discuss possible applications of known Fourier uniformity results for dense finite Sidon sets, to explore possible extensions to weighted sums, and to improve the exposition of some arguments. The author did not use AI-generated text without revision and independently checked all mathematical claims, proofs, and references. The author takes full responsibility for the content of the manuscript.

\section*{Declaration of interests}
The author declares that there is no conflict of interest in this research.

\end{document}